\documentclass[12pt]{article}
\usepackage{graphicx}
\usepackage{amsmath}
\usepackage{amssymb}
\usepackage{theorem}
\usepackage{color}

\numberwithin{equation}{section}

\newtheorem{thm}{Theorem}
\newtheorem{lemma}[thm]{Lemma}
\newtheorem{prop}[thm]{Proposition}
\newtheorem{cor}[thm]{Corollary}
{\theorembodyfont{\rmfamily}

}

\newcommand{\qed}{\hfill \mbox{\raggedright \rule{.07in}{.1in}}}
 
\newenvironment{proof}{\vspace{1ex}\noindent{\bf
Proof}\hspace{0.5em}}{\hfill\qed\vspace{1ex}}

\newcommand{\R}{{\mathbb R}}

\newcommand{\Z}{{\mathbb Z}}

\newcommand{\N}{{\mathbb N}}
\newcommand{\PP}{{\mathbb P}}

\newcommand{\cC}{{\mathcal C}}
\newcommand{\cJ}{{\mathcal J}}
\newcommand{\cM}{{\mathcal M}}

\newcommand{\eps}{\epsilon}

\newcommand{\sgn}{\operatorname{sgn}}

\newcommand{\SMALL}{\textstyle}

\newcommand{\E}{{\bf E}}

\title{Necessary and sufficient condition for $\cM_2$-convergence to a L\'evy process 
\\ for billiards with cusps at flat points}

\author{
Paul Jung    \thanks{Department of Mathematical Sciences, KAIST, Daejeon, South Korea, supported in part by NRF grant N01170220.}
\and
 Ian Melbourne    \thanks{Mathematics Institute, University of Warwick, Coventry, CV4 7AL, UK}
\and
Fran\c{c}oise P\`ene    \thanks{Universit\'e de Brest, Institut Universitaire de France, LMBA, UMR CNRS 6205, 29238 Brest cedex, France}
 \and 
Paulo Varandas\thanks{Departamento de Matem\'atica, Universidade Federal da Bahia,
40170-110 Salvador, Brazil}
\and
Hong-Kun Zhang\thanks{Department of Mathematics, University of Massachusetts, Amherst, MA, USA}
}

\date{27 October 2019}

\begin{document}

 \maketitle

\begin{abstract}
We consider a class of planar dispersing billiards with a cusp at a point of vanishing curvature.  
Convergence to a stable law and to the corresponding L\'evy process in the $\cM_1$ and $\cM_2$ Skorohod topologies has been studied in recent work.  Here we show that certain sufficient conditions for $\cM_2$-convergence are also necessary.
\end{abstract}

 \section{Introduction} 
 \label{sec-intro}

B\'alint, Chernov \& Dolgopyat~\cite{BalintChernovDolgopyat11} proved anomalous diffusion for planar dispersing billiards with cusps, namely that the central limit theorem holds with nonstandard normalization $(n\log n)^{1/2}$ instead of 
$n^{1/2}$.  They also established the functional version, proving weak convergence to Brownian motion (again with the anomalous normalization $(n\log n)^{1/2}$).
Their results hold for all H\"older observables with zero mean, and
throughout we restrict attention to such observables.

Recently, Zhang~\cite{Zhang17a} introduced a class of billiards with cusps where the boundary has vanishing curvature at a cusp.
For definiteness, we suppose throughout that there is a single symmetric cusp; more general examples are considered in~\cite{JungPeneZhangsub}.
Jung \& Zhang~\cite{JungZhang18} proved convergence to
an $\alpha$-stable law (with normalization $n^{1/\alpha}$) for such billiards; any $\alpha\in(1,2)$ can be achieved depending on the flatness at the cusp.  Then in two papers~\cite{JungPeneZhangsub,MVsub} written independently and using different methods, we  obtained functional versions, yielding weak convergence to the corresponding $\alpha$-stable L\'evy process.

For convergence to a L\'evy process, there is the question as to which topology to use on the Skorohod space of c\`adl\`ag paths.
For H\"older (and hence bounded) observables, it is easy to see that convergence in the standard $\cJ_1$ Skorohod topology~\cite{Skorohod56} fails.
Convergence in the $\cM_1$ Skorohod topology is proved
in~\cite{JungPeneZhangsub} for observables that have constant sign near the cusp.  In~\cite{MVsub}, necessary and 
sufficient conditions for $\cM_1$-convergence, and 
sufficient conditions for convergence in the $\cM_2$ topology, are given.
In particular, by~\cite{MVsub} there are observables for which convergence fails in
$\cM_1$ but holds in the weaker $\cM_2$ topology.
Moreover, it is conjectured in~\cite{MVsub} that the 
sufficient conditions for $\cM_2$-convergence are necessary
and hence that there are observables for which convergence fails in all of the Skorohod topologies~\cite{Skorohod56}.
(For more details about the various Skorohod topologies, we refer to~\cite{JungPeneZhangsub,MVsub,MZ15,Skorohod56,Whitt}.)

In this paper, we prove the conjecture in~\cite{MVsub}:
the conditions therein for $\cM_2$-convergence are indeed necessary and sufficient.  To prove this we need extra information from the proof in~\cite{JungPeneZhangsub}.  There it is shown that a certain first return process, denoted $U_n$ in Lemma~\ref{lem:U} below, converges in $\cJ_1$ (the setup in~\cite{MVsub} yields only $\cM_1$-convergence for $U_n$ which seems insufficient for proving necessity of conditions for $\cM_2$-convergence for the full process).

In the remainder of the introduction, we describe the example in
\cite{JungZhang18} and state our main result.
For the sake of simplicity, we concentrate on the setting of a single cusp which is the example studied in~\cite{JungZhang18,MVsub}.   As explained in Section~\ref{sec:several}, our main result extends straightforwardly to more general billiard tables with finitely many cusps with vanishing curvature as studied in~\cite{JungPeneZhangsub}. 

The Jung \& Zhang example~\cite{Zhang17a, JungZhang18} is 
a billiard with a table $Q\subset\R^2$ whose boundary consists  of a finite number of $C^3$ curves $\Gamma_i$, $i=1,\dots,n_0$, where $n_0\ge3$, with a cusp formed by two of these curves $\Gamma_1$, $\Gamma_2$.
 The other intersection points correspond to corners.
In coordinates $(s,z)\in\R^2$, the cusp lies at $(0,0)$ and $\Gamma_1$, $\Gamma_2$ are tangent to the $s$-axis at $(0,0)$.  Moreover, in a small neighborhood of $(0,0)$, the curves $\Gamma_1$ and $\Gamma_2$ can be represented as the graph of 
$z= \beta^{-1}s^\beta$
and $z=- \beta^{-1}s^\beta$
 respectively,  where $\beta>2$ is a parameter.

The phase space of the billiard map (or collision map) $T$ is given by 
$\Lambda=\partial Q\times[0,\pi]$, with coordinates $(r,\theta)$
where $r$ denotes arc length along $\partial Q$
 and $\theta$ is the angle between the tangent line of the boundary and the collision vector in the clockwise direction.
There is a natural ergodic invariant probability measure $d\mu=(2|\partial Q|)^{-1}\sin\theta\,dr\,d\theta$ on~$\Lambda$, where $|\partial Q|$ is the length of $\partial Q$.

In configuration space,
the cusp is a single point $(0,0)=\Gamma_1\cap\Gamma_2$.
Let $r'\in\Gamma_1$ and
$r''\in\Gamma_2$ be the arc length coordinates of $(0,0)$.
Then in phase space $\Lambda$, the cusp is
the union of two line segments
\[
\cC=\{(r',\theta):0\le\theta\le\pi\} \cup\{(r'',\theta):0\le\theta\le\pi\}.
\]
Let $\alpha=\frac{\beta}{\beta-1}\in(1,2)$ where $\beta>2$ is the curvature at the flat cusp as described above.
Given $v:\Lambda\to\R$ continuous, define 
\[
I_v(s)=\frac12\int_0^s \{v(r',\theta)+v(r'',\pi-\theta)\} (\sin\theta)^{1/\alpha}\,d\theta,
\quad  s\in[0,\pi].
\]
Now suppose that $v:\Lambda\to\R$ is a H\"older mean zero observable.
By~\cite{JungZhang18}, 
convergence to a totally skewed $\alpha$-stable law holds provided $I_v(\pi)\neq0$.  We restrict from now on to the case
$I_v(\pi)>0$.
(The case $I_v(\pi)<0$ is similar with obvious modifications.)
 Then
$n^{-1/\alpha}\sum_{j=0}^{n-1}v\circ T^j\to_d G$ where
$G$ has characteristic function  
\[
\E(e^{iuG})=\exp\{-|u|^\alpha\sigma^\alpha(1-i\sgn u\tan{\SMALL\frac{\pi\alpha}{2}})\},
\quad \sigma^\alpha=\frac{I_v(\pi)^\alpha\Gamma(1-\alpha)\cos{\SMALL\frac{\pi\alpha}{2}}}{2^{\alpha-1}\beta|\partial Q|}\, .
\]
We remark that the constant 
$\Gamma(1-\alpha)\cos{\SMALL\frac{\pi\alpha}{2}}$
was missing from the statement of the main result in Jung \& Zhang~\cite{JungZhang18}, but this error was corrected in the papers~\cite{JungPeneZhangsub, MVsub}. There is also a difference in the scale parameter between \cite{JungPeneZhangsub} and \cite{MVsub}, by a factor of $2^{-\alpha}$, due to a difference of a factor of $2$ in the definitions of $I_v(s)$ in those two papers. We use the version of $I_v(s)$ in~\cite{MVsub}.

The references~\cite{JungPeneZhangsub,MVsub} study the corresponding functional limit law $W_n\to_w W$ in $D([0,\infty),\cM_1)$ where
$$W_n(t)=n^{-1/\alpha}\sum_{j=0}^{[nt]-1}v\circ T^j,$$
and $W$ is the $\alpha$-stable L\'evy process with $W(1)=_dG$.
In particular, by~\cite[Theorem~1.3]{MVsub} a necessary and sufficient condition for convergence in $\cM_1$
is that $s\mapsto I_v(s)$ is monotone.
When convergence in $\cM_1$ breaks down,~\cite[Theorem~1.4]{MVsub} gives a sufficient condition for convergence in $\cM_2$, namely that $I_v(s)\in[0,I_v(\pi)]$ for all $s\in[0,\pi]$.  Our main result is that this condition is also necessary:

\begin{thm} \label{thm:M2iff}
Let $v:\Lambda\to\R$ be a H\"older mean zero observable with $I_v(\pi)>0$.  Then $W_n\to_w W$ in $(D[0,\infty),\cM_2)$ as $n\to\infty$ if and only if
$I_v(s)\in[0,I_v(\pi)]$ for all $s\in[0,\pi]$.
\end{thm}

\section{Proof of Theorem~\ref{thm:M2iff}}
\label{sec:proof}

As in~\cite{JungZhang18} and~\cite{MVsub} we consider the 
first return map $f=T^\varphi:X\to X$ where $X$ is a region bounded away from the cusp.
Specifically,
let $X=(\Gamma_3\cup\dots\cup\Gamma_{n_0})\times[0,\pi]$ and define the first return time $\varphi:X\to\Z^+$ and first return map $f=T^\varphi:X\to X$,
\[
\varphi(x)=\inf\{n\ge1:T^nx\in X\}, \qquad f(x)=T^{\varphi(x)}x.
\]
Define $\varphi_k=\sum_{j=0}^{k-1}\varphi\circ f^j$ 
and $v_k=\sum_{j=0}^{k-1}v\circ T^j$
for $k\ge0$.
Also, set $\bar\varphi=\frac 1{\mu(X)}\int_X\varphi\,d\mu$.

\begin{prop} \label{prop:max}
$\lim_{n\to\infty} n^{-1}\max_{j\le n}\varphi\circ f^j=0$ almost surely.
\end{prop}

\begin{proof}
Since $\varphi$ is integrable, it follows from the ergodic theorem that $\lim_{n\to\infty}n^{-1}\varphi_n=\bar\varphi$ a.e.\ and so
$\lim_{n\to\infty}n^{-1}\varphi\circ f^n=0$ a.e.
The result follows easily.
\end{proof}

Turning to the proof of Theorem~\ref{thm:M2iff},
we continue to assume that $v:\Lambda\to\R$ is a H\"older mean zero observable with $I_v(\pi)>0$.
By~\cite{MVsub}, it suffices to consider the case $I_v(s)\not\in[0,I_v(\pi)]$ for some $s$.
From now on, we suppose that 
\[
\max_{s\in[0,\pi]}I_v(s)>I_v(\pi).
\]
(The case 
$\min_{s\in[0,\pi]}I_v(s)<0$,
is treated similarly.)

Let $I_v^*=\max_{s\in[0,\pi]}I_v(s)$ and choose $s^*$ such that
$I_v^*=I_v(s^*)$.
Define
\[
\ell^*:X\to\N, \qquad \ell^*(x)=[\varphi(x)\Psi(s^*)],
\]
where $\Psi:[0,\pi]\to[0,1]$ is the diffeomorphism
$\Psi(s)=I_1(\pi)^{-1}I_1(s)$.

\begin{prop} \label{prop:M2iff}
There exists $C,\delta>0$ such that
\[
\Big|\frac{W_n((\varphi_k+\ell^*\circ f^k)/n)-W_n(\varphi_k/n)}{
W_n(\varphi_{k+1}/n)-W_n(\varphi_k/n)}-
\frac{I_v^*}{I_v(\pi)}\Big|
\le C\varphi^{-\delta}\circ f^k,
\]
for all $k\ge0$, $n\ge1$.
\end{prop}

\begin{proof}
We have
$W_n(\varphi_{k+1}/n)-W_n(\varphi_k/n)  =n^{-1/\alpha}v_\varphi\circ f^k$
for $k\ge0$.
Also, $$W_n((\varphi_k+\ell^*\circ f^k)/n)-W_n(\varphi_k/n)=n^{-1/\alpha}v_{\ell^*}\circ f^k.$$

By~\cite[Proposition~8.1]{MVsub}, there exists $\delta>0$ such that
$$v_{\ell}=\varphi I_1(\pi)^{-1}I_v\circ\Psi^{-1}(\ell/\varphi)+O(\varphi^{1-\delta}),$$ for $0\le\ell\le\varphi$.
In particular, we have
$$
n^{1/\alpha}\big\{W_n(\varphi_{k+1}/n)-W_n(\varphi_k/n) \big\}
 =\varphi\circ f^k I_1(\pi)^{-1}I_v(\pi)+O(\varphi^{1-\delta}\circ f^k),$$
\begin{align*}
 n^{1/\alpha}\big\{W_n((\varphi_k&+\ell^*\circ f^k)/n)-W_n(\varphi_k/n)\big\}  
\\
&=
\varphi\circ f^k I_1(\pi)^{-1}
I_v\circ\Psi^{-1}((\ell^*/\varphi)\circ f^k)
+O(\varphi^{1-\delta}\circ f^k).
\end{align*}
Since $I_v\circ\Psi^{-1}$ is $C^1$,
\begin{align*}
I_v\circ\Psi^{-1}(\ell^*/\varphi)
&=I_v\circ\Psi^{-1}([\varphi\Psi(s^*)]/\varphi)\\
&=I_v\circ\Psi^{-1}(\Psi(s^*))+O(\varphi^{-1})
=I_v^*+O(\varphi^{-1}).
\end{align*}
Hence we obtain
\[
n^{1/\alpha}\big\{W_n((\varphi_k+\ell^*\circ f^k)/n)-W_n(\varphi_k/n)\big\}= 
\varphi\circ f^k I_1(\pi)^{-1}I_v^*
+O(\varphi^{1-\delta}\circ f^k)
\]
and the result follows.
\end{proof}

\begin{cor} \label{cor:M2iff}
Let $\eps>0$.
There exists $n_0\ge1$ such that for $k\ge0$, $n\ge n_0$, if
$W_n(\varphi_{k+1}/n)-W_n(\varphi_k/n)\ge 1$,
then
\begin{align*}
& \frac{W_n((\varphi_k+\ell^*\circ f^k)/n)-W_n(\varphi_k/n)}
{W_n(\varphi_{k+1}/n)-W_n(\varphi_k/n)}
 \in\Big[
\frac{I_v^*}{I_v(\pi)}-\eps,
\frac{I_v^*}{I_v(\pi)}+\eps \Big].
\end{align*}
\end{cor}
\begin{proof}
We have
\[
1\le W_n(\varphi_{k+1}/n)-W_n(\varphi_k/n)=n^{-1/\alpha}v_\varphi\circ f^k\le n_0^{-1/\alpha}|v|_\infty \varphi\circ f^k.
\]
This implies that  $$\varphi\circ f^k\ge |v|_\infty^{-1}  n_0^{1/\alpha}.$$
Hence, we can choose $n_0$ so large that $C\varphi^{-\delta}\circ f^k\le\eps$.
The result now follows from Proposition~\ref{prop:M2iff}.
\end{proof}

Following~\cite{MZ15} (see also~\cite[Section~4]{MVsub}),
we write $W_n = U_n + R_n$, where
\[
U_n(t)  = n^{-1/\alpha}\sum_{j=0}^{N_{[nt]}-1}
 v_\varphi \circ f^j   \quad\text{and}\quad 
        R_n(t) = n^{-1/\alpha} \Bigg( \sum_{\ell=0}^{[nt]-  \varphi_{N_{[nt]}}-1} v \circ T^\ell \Bigg) \circ f^{N_{[nt]}}.
\]

Here, $N_k(x)=\max \{ \ell \ge 1 : \varphi_{\ell}(x) \le k \}$ is the number of returns of $x$
to the set $X$, under iteration by the map $T$,
up to time $k$.

\begin{lemma} \label{lem:U} 
  $U_n\to_w W$ in $D([0,\infty),\cJ_1)$
as $n\to\infty$.
\end{lemma}

\begin{proof}
Define the induced process 
\[
\widetilde W_n(t)=n^{-1/\alpha}\sum_{j=0}^{[nt]-1}
v_\varphi \circ f^j,\;\, n\geq 1.
\]
Proceeding as in~\cite[Lemma~3.4]{MZ15}, we note that $U_n=\widetilde W_n\circ g_n$,
where $g_n(t)=n^{-1}N_{[nt]}$.  

Now, by~\cite[Theorem~3.1]{JungPeneZhangsub}, $\widetilde W_n\to_w \bar\varphi^{1/\alpha}W$ in $D([0,\infty),\cJ_1)$.
Also, $g_n\to g$ uniformly on compact subsets of $[0,\infty)$ where $g(t)=t/\bar\varphi$.

Let $D_0$ be the space of elements of $D[0,\infty)$ that are nonnegative and nondecreasing.   Then $(\widetilde W_n,g_n)\in D[0,\infty)\times D_0$.
Since $g$ is continuous and deterministic, it follows from~\cite[Theorem~3.1]{Whitt80} that $$U_n=\widetilde W_n\circ g_n\to_w \bar\varphi^{1/\alpha}W\circ g=W$$
in $D([0,\infty,\cJ_1)$.~
\end{proof}

Given $u\in D[0,1]$, we define $\Delta u(t)=u(t)-u(t-)$.
Let $\Pi_\alpha$ denote the L\'evy measure with density $\alpha x^{-(\alpha+1)}1_{(0,\infty)}(x)$.
Then 
$\#\{t\in[0,1]:\Delta W(t)\in B\}$
has a Poisson distribution with mean $\Pi_\alpha(c^{-1/\alpha}B)$
for each open interval $B\subset(0,\infty)$
(see for example~\cite[Chapter~4]{Sato99}).
Here, $c>0$ is a scaling constant determined by $G$ (and hence $I_v(\pi)$, $\alpha$ and $|\partial Q|$).

For $b>0$, define
\[
E(b) =\{u\in D[0,1]:\Delta u(t)>b\;\text{for some $t\in[0,1]$}\}.
\]
By the above discussion,
$\PP(W\in E(b))=1-e^{-cb^{-\alpha}}$.
Also $U_n\to_w W$ in $\cJ_1$ by Lemma~\ref{lem:U}, so
$$\lim_{n\to\infty}\mu(U_n\in E_n(b))=1-e^{-cb^{-\alpha}}.$$

Similarly, if we suppose for contradiction that
$W_n\to_w W$ in $\cM_2$, then
for any $\eps>0$, $b>0$, there exists $\delta>0$, $n_0\ge1$, such that
for $n\ge n_0$,
\[
\mu\big\{W_n(t)-W_n(t')>b\;\text{for some $0\le t'<t<(t'+\delta)\wedge1$}\big\}<1-e^{-cb^{-\alpha}} + \eps.
\]

Now, $U_n\in E(1)$ if and only if 
$W_n(\varphi_{k+1}/n)-W_n(\varphi_k/n)>1$ for some $0\le k\le N_{[nt]}$.
By Corollary~\ref{cor:M2iff},
this implies for $n$ sufficiently large that
$$W_n((\varphi_k+\ell^*\circ f^k)/n)-W_n(\varphi_k/n)\ge \frac{I_v^*}{I_v(\pi)}-\eps,$$
for some $0\le k\le N_{[nt]}$.  Hence we obtain
that $$W_n(t)-W_n(t')>\frac{I_v^*}{I_v(\pi)}-\eps,$$
for some $0\le t'<t\le1$ with $t-t'<n^{-1}\max_{k\le n}\varphi\circ f^k$.

Putting these observations together, we obtain
\begin{align*}
1-e^{-c} & = \PP(W\in E(1))=\lim_{n\to\infty}\mu(U_n\in E(1))
\\ & \le \lim_{n\to\infty}\mu\Big\{W_n(t)-W_n(t')>\frac{I_v^*}{I_v(\pi)}-\eps\;\text{for some $0\le t'<t<(t'+\delta)\wedge1$}\Big\}
\\ & \qquad\qquad +\lim_{n\to\infty}\mu\big(n^{-1}\max_{k\le n}\varphi\circ f^k>\delta\big)
\\ & <1-\exp
\Big\{-c
\Big(\frac{I_v^*}{I_v(\pi)}-\epsilon\Big)^{-\alpha}
\Big\} 
+ \eps + \lim_{n\to\infty} \mu\big(n^{-1}\max_{k\le n}\varphi\circ f^k>\delta\big)
.
\end{align*}
By Proposition~\ref{prop:max}, 
$1-e^{-c} \le 1-\exp\big\{-c
\big(\frac{I_v^*}{I_v(\pi)}-\epsilon\big)^{-\alpha} \big\} +\eps$.
Also, $\epsilon$ is arbitrary so
$$1-e^{-c} \le 1-\exp\big\{-c
\big(\frac{I_v(\pi)}{I_v^*}\big)^{\alpha} \big\}.$$
Since $I_v^*>I_v(\pi)$, 
this is the desired contradiction and the proof of Theorem~\ref{thm:M2iff} is complete.

\section{Billiards with several cusps}
\label{sec:several}

Our main result, Theorem~\ref{thm:M2iff}, is formulated for the case of a single cusp but extends straightforwardly to billiards with several cusps with vanishing curvature as studied in~\cite{JungPeneZhangsub}.
In this section, we sketch the arguments for proving this extended result.

In particular, in   \cite{JungPeneZhangsub}, the billiard table $Q$ is such that there are multiple cusps with bounding curves locally of the form
 $$\Gamma_{i}=\{(s,C_is^{\beta_i}+{\cal O} (s^{2\beta_i-1}))\},\quad
\Gamma_i'=\{(s,-C'_{i}s^{\beta_i}+{\cal O} (s^{2\beta_i-1}))\},$$ 
with tangent vectors $(1,\beta_i C_is^{\beta_i-1}+{\cal O} (s^{2\beta_i-2}))$
and  $(1,-\beta_i C'_is^{\beta_i-1}+{\cal O} (s^{2\beta_i-2}))$,
where $\max_i\beta_i>2$ and where $C_i>0$ and $C'_i\ge 0$.
It is assumed that any two cusps are not diametrically opposite, in the sense that the tangent trajectory coming out of a cusp does not match up precisely with the tangent trajectory coming out of any another cusp. 
Let $\beta=\max_i\beta_i$.  The cusps with $\beta_i<\beta$ play no role in the subsequent analysis.

For each $i$ with $\beta_i=\beta$, define
\[
I_v^{(i)}(s)=\frac12\int_0^s \{v(r_{i}',\theta)+v(r_{i}'',\pi-\theta)\} (\sin\theta)^{1/\alpha}\,d\theta,
\quad  s\in[0,\pi],
\]
where $\alpha=\frac{\beta}{\beta-1}$.

Following~\cite{JungPeneZhangsub}, we consider the first return map to a region $X$ bounded away from all the cusps.  The induced process $\widetilde W_n$ defined as in the proof of Lemma~\ref{lem:U} converges to an $\alpha$-stable L\'evy process in the $\cJ_1$ topology by~\cite[Theorem~3.1]{JungPeneZhangsub}.  
In addition~\cite[Theorem~2.2]{JungPeneZhangsub} gives sufficient conditions for convergence of the 
full process $W_n$ in the $\cM_1$ topology to a rescaled $\alpha$-stable L\'evy process $W$ (we refer to~\cite{JungPeneZhangsub} for the definition of $W$).

We can now formulate necessary and sufficient conditions for convergence \mbox{$W_n\to_w W$} in the $\cM_1$ and $\cM_2$ topology.
By~\cite{MVsub}, convergence in $\cM_1$ holds if and only if $s\mapsto I_v^{(i)}(s)$ is  monotone for each $i$.
By~\cite{MVsub} and the argument in Section~\ref{sec:proof} of the current paper, convergence in $\cM_2$ holds if and only if $I_v^{(i)}(s)$ lies between $0$ and $I_v^{(i)}(\pi)$ for all $s\in[0,\pi]$ and for each $i$.

\end{document}